\documentclass[12pt]{article}
\usepackage{theorem,amsmath,amssymb,latexsym, amscd}
\usepackage[all]{xy}

\newtheorem{theorem}{Theorem}[section]
\newtheorem{proposition}[theorem]{Proposition}
\newtheorem{lemma}[theorem]{Lemma}
\newtheorem{corollary}[theorem]{Corollary}

\newtheorem{conjecture}{Conjecture}[section]

\theorembodyfont{\upshape}
{\theoremstyle{break}
\newtheorem{remarks}[theorem]{Remarks}
}

\theorembodyfont{\upshape}
\newtheorem{remark}[theorem]{Remark}

\theorembodyfont{\upshape}

\theorembodyfont{\upshape} 
\newtheorem{definition}[theorem]{Definition}

\theorembodyfont{\upshape}
\newtheorem{question}[theorem]{Question}

\newtheorem{problem}[theorem]{Problem}

\newcommand{\alg}{{\text{alg}}}

\newcommand{\homo}{{\text{hom}}}
\newcommand{\Hom}{\operatorname{Hom}}
\newcommand{\End}{\operatorname{End}}

\newcommand{\Gal}{\operatorname{Gal}}
\newcommand{\Pic}{\operatorname{Pic}}

\newcommand{\Tor}{{\operatorname{Tor}}}

\newcommand{\Spec}{\operatorname{Spec}}

\newcommand{\Br}{\operatorname{Br}}

\newcommand{\colim}{\operatornamewithlimits{colim}}
\newcommand{\rk}{\operatorname{rank}}

\newcommand{\NS}{\operatorname{NS}}

\newcommand{\Z}{{{\mathbb Z}}}
\newcommand{\Q}{{{\mathbb Q}}}
\newcommand{\R}{{{\mathbb R}}}

\newcommand{\T}{{{\mathcal T}}}

\newcommand{\M}{{\mathcal M}}

\newcommand{\et}{{\text{\rm et}}}

\newcommand{\ab}{{\text{\rm ab}}}
\newcommand{\Zar}{{\text{\rm Zar}}}

\newcommand{\Alb}{\operatorname{Alb}}

\theorembodyfont{\itshape}

\theorembodyfont{\upshape}

\theorembodyfont{\upshape} 
\newtheorem{note}{Note.}

\theorembodyfont{\upshape}

\theorembodyfont{\upshape}
\newtheorem{exercise}{Exercise.}

\theorembodyfont{\upshape}
\newtheorem{example}{Example.}

\theorembodyfont{\upshape}
\newtheorem{proof}{Proof.}

\theorembodyfont{\upshape}

\newcommand{\qed}{\nopagebreak\par\hspace*{\fill}$\square$\par\vskip2mm}
\newcommand{\myemail}[1]{\indent \emph{E-mail:} {\tt #1}}
\newcommand{\myaddress}[1]{\indent {\sc #1}\par}

\title{Duality of integral \'etale motivic cohomology}
\author{Thomas H.~Geisser}
\date{}

\begin{document}
\maketitle

\begin{abstract}
We discuss duality pairings on integral \'etale motivic cohomology groups
of regular and proper schemes
over algebraically closed fields, local fields, finite fields, and 
arithmetic schemes.
\end{abstract}

\section{Introduction}
Let $X$ be a smooth and projective
variety of dimension $d$ over a perfect field $k$. 
Using duality
theorems for \'etale cohomology with finite coefficients, we show
duality results on integral \'etale motivic cohomology groups.

If $k$ is algebraically closed, and $m$ is an integer prime to the 
characteristic of $k$, we construct for all $n,u$ satisfying $w=n+u-d>0$ 
a Galois invariant pairing 
\begin{equation}\label{cpair1}
H^{2d+1-i}_\et(X,\Z(u))/m \times {}_m H^{i}_\et(X,\Z(n)) 
\to \Z/m(w) 
\end{equation}
which is non-degenerate on the left. For the right kernel
${}_m H^{i}_\et(X,\Z(n))^0$, we obtain a secondary perfect pairing 
\begin{equation}\label{secondary}
{}_m H^{i}_\et(X,\Z(n))^0\times {}_m H^{2d+2-i}_\et(X,\Z(u))^0
\to \Z/m(w), 
\end{equation}
and we show that for $n+u=d+1$, $i=2n$ and $2d+2-i=2u$,
this pairing becomes perfect when restricted to the divisible subgroups. 
Thus the pairing can be thought of as a
generalization of the $e_m$-pairing between the Picard and Albanese abelian 
variety of $X$. 

If $k$ is finite, we construct for all $m$ pairings
\begin{equation}
H^{2d+2-i}_\et(X,\Z(d-n))/m \times {}_m H^{i}_\et(X,\Z(n)) \to \Z/m 
\end{equation}
which are non-degenerate on the left. We conjecture that the pairing is perfect, 
and relate this conjecture to Tate's conjecture on the surjectivity
of the cycle map.

If $k$ is local and $m$ prime to the characteristic of $k$,
we construct pairings 
$$H^{2d+3-i}_\et(X,\Z(d+1-n))/m \times  {}_m H^{i}_\et(X,\Z(n))\to\Z/m$$
which are non-degenerate on the left. In case $d=0$, the perfectness
of the pairing is equivalent to the statement of local class field theory,
and for $d=n=1$ it amounts to Lichtenbaum's duality between the Picard
group and the Brauer group for curves over a $p$-adic field 
\cite{lichtenbaumcurve}. 
In contrast, we show that the pairing can have a right kernel 
for curves, and we give an example showing that the pairing
can have a right kernel even in the good reduction case.
In particular, there is no duality "in some appropriate sense
of the term" expected by Lichtenbaum \cite[\S 6]{lichtenbaumMC}.

\medskip

Notation: 
We denote Bloch's motivic complex by $\Z(n)$, a complex
of \'etale sheaves \cite[Lemma 3.1]{ichdede}. When we need to 
emphasize that $\Z(n)$ is considered as a complex of Zariski sheaves,
we write $\Z(n)^\Zar$. 

For an abelian group $A$,  we denote by ${}_m A$ its $m$-torsion,
by $A\{l\}=\colim_{r} {}_{l^r}A$ its subgroup of
$l$-power torsion elements, by $A^*$ its Pontrjagin dual
$\Hom(A,\Q/\Z)$, by $A^\wedge=\operatorname{lim}_m A/m$ its completion,
by $A^{\wedge l}=\operatorname{lim}_r A/l^r$ its the $l$-adic completion, and by
$T_lA=\operatorname{lim}_r {}_{l^r} A$ its $l$-adic Tate module.

\section{Algebraically closed fields}
Over an algebraically closed base field, Zariski and \'etale hypercohomology
of the motivic complex  agree in weights at least the dimension, 
i.e. if $\epsilon: X_\et\to X_\Zar$ is
the change of topology map, then the adjunction map 
$\Z(n)^\Zar\to R\epsilon_*\Z(n)$ is a quasi-isomorphism of complexes
of sheaves for the Zariski-topology for $n\geq d$.
This was deduced in \cite{ichduality} from a theorem of Suslin away from
the characteristic and from \cite{ichmarc} at the characteristic. 
Since the Zariski hypercohomology admits a push-forward map for the proper map
$f:X\to k$,
we obtain a Gysin map for \'etale motivic cohomology as the composition
$$Rf_*R\epsilon_* \Z(w+d)_X \cong Rf_*\Z(w+d)_X^\Zar\to \Z(w)_k^\Zar
\cong R\epsilon_*\Z(w)_k.$$ 
If $k$ is a perfect field, then applying Galois cohomology $\Gamma(\Gal(k),-)$ 
to this Gysin map over the base extension to the algebraic closure, 
we obtain a "trace" map
$$H^{2d+v}_\et(X,\Z(w+d))\to H^{v}_\et(k,\Z(w))$$
for $X$ proper over $k$ and any $w\geq 0$.

The cup product pairing on higher Chow groups induces a
product on \'etale hypercohomology, hence for $w=n+u-d$ we obtain a pairing 
\begin{equation}\label{basicpair}
 H^{2d+v-i}_\et(X,\Z(u))\times H^{i}_\et(X,\Z(n)) \to 
H^{2d+v}_\et(X,\Z(u+n)) 
\stackrel{tr}{\to} H^v_\et(k,\Z(w)).
\end{equation}
If $k$ is algebraically closed, $m$ invertible in $k$, and $w\geq 1$, then
the coefficient sequence gives an isomorphism
$$\Q/\Z(w)\cong H^0_\et(k,\Q/\Z(w))\cong \Tor H^1_\et(k,\Z(w)).$$
Indeed, this follows by comparing to Suslin's calculation 
of the $K$-theory of an algebraically closed field.
Restricting the pairing \eqref{basicpair} for $v=1$ to the $m$-torsion on the
right, we obtain for 
$$ n+u>d, \quad i+j=2d$$
a pairing
\begin{equation*}
H^{j}_\et(X,\Z(u))/m \times {}_m H^{i+1}_\et(X,\Z(n)) \to  
{}_m H^{2d+1}_\et(X,\Z(u+n)) \to \Z/m(w) .
\end{equation*}
This is compatible with Poincare-duality for \'etale cohomology
with finite coefficients
\begin{equation}\label{ababa}
\begin{CD}
H^{j}_\et(X,\Z(u))/m @.\times @. {}_m H^{i+1}_\et(X,\Z(n)) @>>>
 {}_mH^1_\et(k,\Z(w)) \cong \Z/m(w) \\
@VVV @.@A\partial AA @A\partial A\cong A \\
H^{j}_\et(X,\Z/m(u))@.\times @. H^{i}_\et(X,\Z/m(n)) @>>> 
H^0_\et(k,\Z/m(w)) \cong \Z/m(w)
\end{CD}\end{equation}
because $\partial(x\cup y)=x\cup \partial(y) $ if $\partial(x)=0$.
Consequently, we obtain a map of short exact sequences of finite abelian groups
\begin{equation} \label{cdagr}
\begin{CD}
H^{j}_\et(X,\Z(u))/m@>>>  H^{j}_\et(X,\Z/m(u))@>>> 
{}_m H^{j+1}_\et(X,\Z(u))\\
@VVV @VVV @VVV \\
{}_m H^{i+1}_\et(X,\Z(n))^\sharp@>>> H^{i}_\et(X,\Z/m(n))^\sharp@>>> 
(H^{i}(X,\Z(n))/m)^\sharp,
\end{CD}\end{equation}
where for any abelian group $A$, $A^\sharp=\Hom(A,\Q/\Z(w))$ is the twist the usual Pontrjagin dual.
The middle map is an isomorphism by Grothendieck's
Poincar\'e duality for \'etale cohomology, so that 
the snake Lemma gives an exact sequence
\begin{multline}\label{cseq}
 0\to H^{j}_\et(X,\Z(u))/m \to ({}_m H^{i+1}_\et(X,\Z(n)))^\sharp \\
\stackrel{\delta}{\longrightarrow}
 {}_m H^{j+1}_\et(X,\Z(u))\to (H^{i}_\et(X,\Z(n))/m)^\sharp\to 0.
\end{multline}
In particular, the pairings \eqref{cpair1} are non-degenerate
on the left, and $\delta$ induces the pairing \eqref{secondary}.
It is easy to see that 
the sequence \eqref{cseq} is compatible with varying $m$.

\begin{remark}
The construction of the pairing requires $u+n\geq d+1$.
For example, for $X=\Spec k$ and $u=n=0$, the
diagram \eqref{ababa} does not commute. The construction also does not
work for $m$ a power of the characteristic of the base field, because
then $\Z/m(w)=0$ for $w>0$.
\end{remark}

\subsection*{The case $u+n=d+1, i=2n-1, j=2u-1$}

\begin{example}(Rojtman's theorem)
If $n=d, u=1, i=2d-1, j=1$, then $H^{j}_\et(X,\Z(u))/m= k^\times /m =0$.
Moreover, it was shown in \cite{ichstructure} that 
$H^{2d-1}_\et(X,\Z(d))$ modulo its divisible subgroup is isomorphic to the
dual of $\Tor \NS X$. Hence  
the fact that $(A^\sharp/m)^\sharp={}_m A$ for a finite group $A$ imply that  
\eqref{cseq} gives a short exact sequence
$$ 0\to {}_m H^{2d}_\et(X,\Z(d))^\sharp \to {}_m \Pic X\to {}_m\NS X\to 0.$$
From this we can deduce Rojtman's theorem away from the 
characteristic. Indeed, since
$CH_0(X)\cong H^{2d}_\et(X,\Z(d))$ as explained in the beginning of
this section, and the Albanese map 
$CH_0(X)\to \Alb_X(k)$ is surjective, it suffices to show
that the order of the $m$-torsion of both sides agrees.
But from the duality of the Picard and Albanese variety we know
that $|{}_m\Alb_X(k)|= |{}_m\Pic^0_X(k)|= |{}_m \Pic X|/| {}_m\NS X|$.
\end{example}

\begin{proposition} Let $u+n=d+1$ and 
assume that $H^{2u-1}_\et(X,\Z_l)$ and $H^{2n-1}_\et(X,\Z_l)$
are torsion free. Then we have a perfect pairing
$$ {}_{l^r} H^{2n}_\et(X,\Z(n))\times {}_{l^r} H^{2u}_\et(X,\Z(u))
\to \mu_{l^r}.$$
\end{proposition}

\begin{proof}
This follows because $H^{2u-1}_\et(X,\Z(u))$ is the extension of an $l$-divisible
group by a finite group contained in  $H^{2u-1}_\et(X,\Z_l)$, see \cite{ichstructure}.
It follows that under the hypothesis the outer terms in \eqref{cseq} vanish.
\qed
\end{proof}

Note that the $l$-adic cohomoloy groups in question are torsion free 
for almost all $l$,
and they are torsion free for all $l$ if $X$ is an abelian variety.
To get an unconditional pairing, consider the subgroup of divisible elements
\begin{align*}
H^{2n}_\homo(X,\Z(n))&=
\ker H^{2n}_\et(X,\Z(n))\to \operatorname{lim}_m H^{2n}_\et(X,\Z(n))/m\\
&= \ker H^{2n}_\et(X,\Z(n))\to \prod_l H^{2n}_\et(X,\Z_l(n)).
\end{align*}
By \cite[Cor. 2.2]{ichstructure}, $\Tor H^{2n}_\et(X,\Z(n))$ is a direct summand
of $H^{2n}_\et(X,\Z(n))$. This property is shared with its subgroup of 
divisible elements $H^{2n}_\homo(X,\Z(n))$, and from this
one easily concludes that $H^{2n}_\homo(X,\Z(n))$ is in fact the 
maximal divisible subgroup of $H^{2n}_\et(X,\Z(n))$.

\begin{proposition}
We have a perfect pairing 
$$ {}_{m} H^{2n}_\homo(X,\Z(n))\times {}_{m} H^{2u}_\homo(X,\Z(u))
\to \mu_{m}.$$
\end{proposition}

\begin{proof}
We can assume that $m=l^r$ is a prime power. Let 
$$Q=\big(H^{2n}_\et(X,\Z(n))/H^{2n}_\homo(X,\Z(n)\big)\otimes\Z_{(l)}\subseteq 
\operatorname{lim}_r H^{2n}_\et(X,\Z(n))/l^r\subseteq H^{2n}_\et(X,\Z_l(n))$$
be the canonical inclusion. Both inclusions are isomorphisms on
torsion subgroups: The former because the cokernel of the completion map
$A\to \operatorname{lim}_r A/l^r$ is uniquely divisible, and 
the latter because the Tate-module $T_l H^{2n+1}_\et(X,\Z(n))$
is torsion free. 

Since $H^{2n}_\homo(X,\Z(n))$ is divisible, we obtain a short 
exact sequence 
$$0\to ({}_m Q)^\sharp \to ({}_m H^{2n}_\et(X,\Z(n)))^\sharp \to 
({}_m H^{2n}_\homo(X,\Z(n)))^\sharp \to  0 .$$
The group $H^{2u-1}_\et(X,\Z(u))$ is an extension of a finite group by a divisible
group \cite[Thm. 1.1]{ichstructure}, 
and the $l$-primary part of this finite group is 
$H^{2u-1}_\et(X,\Z_l(u))\{l\}$, which is dual to $H^{2n}_\et(X,\Z_l(n))\{l\}$,
\cite[Prop. 1.2 (1)]{ichstructure}, the $l$-primary part of $Q$ by the above.
We conclude that in the sequence \eqref{cseq}, 
the image of $H^{2u-1}_\et(X,\Z(u))/m$
in $({}_m H^{2n}_\et(X,\Z(n)))^*$ is exactly $({}_m Q)^*$.\qed
\end{proof}
\begin{example}
We have $H^{2}_\homo(X,\Z(1))\cong \Pic^0(X)$, and 
$H^{2d}_\homo(X,\Z(d))\cong CH_0(X)^0$,  and 
obtain another proof of Rojtman's theorem.
\end{example}

A class is algebraically equivalent to zero if it lies in the 
image of some map
$$ H^{2n}_\et(T\times X,\Z(n)) \stackrel{t_1^*-t_0^*}{\longrightarrow} 
H^{2n}_\et(X,\Z(n)) $$
for a smooth connected scheme $T$ (which we can assume to be a smooth
curve) and closed points $t_0,t_1\in T$. The subgroup of classes
algebraically equivalent to zero is written $H^{2n}_\alg(X,\Z(n))$.
It is a subgroup of $H^{2n}_\homo(X,\Z(n))$, hence we can restrict the
pairing above.

\begin{definition}
A homomorphism from $H^{2n}_\alg(X,\Z(n))$ to the $k$-rational points of an abelian
variety $A$ is regular, if for every pointed smooth connected 
variety $t_0\in T$ and element $\Gamma\in H^{2n}_\alg(T\times X,\Z(n))$, 
the composition with
$$T(k)\to H^{2n}_\alg(X,\Z(n)), \qquad t\mapsto t^*\Gamma-t_0^*\Gamma$$
is the map induced on closed points by a morphism of varieties $T\to A$.
\end{definition}

In \cite{murrebar}, \cite{murretor}, Murre studied the situation
for Chow groups, and he proved that a universal homomorphism to an abelian
variety exists for dimension $0$, and codimensions $1$ and $2$.

\begin{theorem}\label{repr}
There is a universal object $\rho_n:H^{2n}_\alg(X,\Z(n))\to A_n$ 
for regular homomorphisms from $H^{2n}_\alg(X,\Z(n))$ to abelian varieties.
\end{theorem}

\begin{proof}
This follows by the argument of Serre-H.\ Saito \cite{saito} 
because the dimension of surjective maps to abelian varieties is 
bounded, see also \cite{korita}.\qed
\end{proof}

\begin{question}
Is there a duality between the abelian varieties 
$A_n$ and $A_{u}$ of Theorem \ref{repr}
such that the diagram below arising from the $e_m$-pairing is commutative? 
$$\begin{CD}
{}_m H^{2n}_\alg(X,\Z(n))@. \times @.\ {}_m H^{2u}_\alg(X,\Z(u))
 @>>> \mu_m\\ 
@V\rho_n VV @. @V\rho_{u}VV @| \\
{}_m A_n@. \times @. {}_m A_{u} @>>> \mu_m.
\end{CD}$$
\end{question}

\section{Finite fields}
Over a finite field, the pairing \eqref{basicpair} for $v=2$ and $w=0$
becomes for 
$$u+n=d,\quad  i+j=2d+1$$
the pairing
$$ H^{j}_\et(X,\Z(u))\times  H^{i+1}_\et(X,\Z(n))\to H^{2d+2}_\et(X,\Z(d)) 
\stackrel{tr}{\to} H^2_\et(k,\Z) \cong \Q/\Z.$$
This is compatible with Poincar\'e duality 
$$ H^{i}_\et(X,\Z/m(n))\times H^{j}_\et(X,\Z/m(u))\to \Q/\Z$$
for all integers $m$
(the pairing for $m$ a power of $p$ is discussed in \cite{milneppart}
using the isomorphism $\Z/m(n)\cong \nu_r(n)$ from \cite{ichmarc}),  
i.e. the diagram \eqref{ababa} commutes. We obtain as in the previous case 
an exact sequence
\begin{multline}\label{seq1}
 0\to H^{j}_\et(X,\Z(u))/m\to ({}_m H^{i+1}_\et(X,\Z(n)))^*\\ 
\stackrel{\delta}{\longrightarrow} {}_m H^{j+1}_\et(X,\Z(u))
\to (H^{i}_\et(X,\Z(n))/m)^*\to 0,
\end{multline}
hence the resulting pairing
\begin{equation}\label{cpair7}
H^{j}_\et(X,\Z(u))/m \times {}_m H^{i+1}_\et(X,\Z(n)) \to \Z/m 
\end{equation}
is non-degenerate on the left. It is non-degenerate if and only if 
$\delta$ vanishes,
and we show that this is equivalent to Tate's conjecture:



\begin{theorem}\label{delzero}
The map $\delta$ vanishes for $i\not=2n, 2n+1$. It vanishes for $i=2n$
if and only if $H^{2n+1}_\et(X,\Z(n))$ is finite. It vanishes
for $i=2n+1$ if and only if $H^{2u+1}_\et(X,\Z(u))$ is finite.
\end{theorem}

Note that $i=2n\Leftrightarrow j=2u+1$ and $i=2n+1\Leftrightarrow j=2u$.
\begin{proof}
We want to show that the map $\delta_m$ in the following diagram is the zero map:
$$\begin{CD}
 (\Tor  H^{i+1}_\et(X,\Z(n))) ^*
 @>\delta_\infty>>  TH^{j+1}_\et(X,\Z(u))\\
 @VvVV @VVV\\
 ({}_m H^{i+1}_\et(X,\Z(n)))^*@>\delta_m >> 
{}_m H^{j+1}_\et(X,\Z(u)).
\end{CD}$$
Since $v$ is surjective, the vanishing of $\delta_\infty$ is equivalent
to the vanishing of $\delta_m$ for all $m$.
The first statement of the Proposition follows because 
$H^{j}_\et(X,\Q/\Z(u))$ is finite for $j\not=2u,2u+1$ for weight reasons, 
and this group surjects onto $\Tor H^{j+1}_\et(X,\Z(u))$, so that
$TH^{j+1}_\et(X,\Z(u))=0$ for $j\not=2u,2u+1$. 

If $i=2n$, then finiteness of $H^{2n+1}_\et(X,\Z(n))$ implies that the
source of $\delta_\infty$ is finite, hence cannot map
non-trivially to a Tate-module. Conversely, if $\delta_m=0$ for all $m$, then 
the duality between the two cotorsion groups 
$H^{2u+1}_\et(X,\Z(u))/m$ and  ${}_m H^{2n+1}_\et(X,\Z(n))$ 
implies that $H^{2n+1}_\et(X,\Z(n))$ is finite.
Reversing the roles of $i$ and $j$ we obtain the result for $i=2n+1$.\qed
\end{proof}

The connection to Tate's conjecture is given by the following (well-known) Proposition.

\begin{proposition}
The finiteness of $H^{2n+1}_\et(X,\Z(n))$ is equivalent to Tate's conjecture 
on the surjectivity of the cycle map in degree $n$ for $X$.
\end{proposition}

\begin{proof}
Consider the coefficient sequence
$$ 0\to H^{2n}_\et(X,\Z(n))^{\wedge l}\to H^{2n}_\et(X,\Z_l(n))
\to T_l H^{2n+1}_\et(X,\Z(n))\to 0.$$
The middle group surjects onto $H^{2n}_\et(\bar X,\Z_l(n))^G$
with finite kernel. On the other hand, in the composition
$$ CH^n(X)\otimes \Z_l \to H^{2n}_\et(X,\Z(n))\otimes \Z_l\to 
H^{2n}_\et(X,\Z(n))^{\wedge l}$$
the left map is an isomorphism up to torsion, and the right map is
surjective (as the target is a subgroup  of $H^{2n}_\et(X,\Z_l(n))$,
hence a finitely generated $\Z_l$-module).
We conclude that the cycle map 
$CH^n(X)\otimes \Z_l \to H^{2n}_\et(\bar X,\Z_l(n))^G$ is 
rationally surjective if and only if $T_l H^{2n+1}_\et(X,\Z(n))=0$
if and only if (the group of cofinite type) $H^{2n+1}_\et(X,\Z(n))\{l\}$
is finite. Finally, by Gabber's theorem \cite{gabber}, the cotorsion of
$H^{2n+1}_\et(X,\Z(n))\{l\}$ vanishes for almost all $l$.\qed
\end{proof}

\begin{example}
If $X$ is a surface, we obtain $H^{6}_\et(X,\Z(1))=0$,
and pairings
\begin{align*}
k^\times &\times   H^{5}_\et(X,\Z(1))\to \Q/\Z\\
{}_m \Pic(X)&\times H^{4}_\et(X,\Z(1))/m\to \Z/m\\
\Pic(X)/m&\times  {}_m H^{4}_\et(X,\Z(1))\to \Z/m\\  
\Br(X)/m&\times  {}_m \Br(X)\to \Z/m.
\end{align*}
The pairings are all perfect, except for the last one, which is perfect
if and only if the Brauer group is finite. In \cite[Thm. 5.1]{tate}, Tate
defines a skew-symmetric pairing on the Brauer group whose kernel consists
exactly of the divisible elements. It is easy to see from the construction
that Tate's pairing is obtained by composing the above pairing with
the canonical map $\Br(X) \to \Br(X)/m$. 
In the limit, this become the composition
$$ \Br(X) \to \Br(X)^\wedge \to \Br(X)^*.$$
The first map has kernel exactly the divisible elements and the second map
is injective.
\end{example}

\begin{remark}
A small modification of \'etale motivic cohomology
yields Weil-\'etale motivic cohomology groups $H^i_W(X,\Z(n))$ 
which are expected  to be finitely generated for all $i,n$ and
smooth and projective $X$ \cite{ichweil}. 
Assuming finite generation, they satisfy dualities
\begin{align*}
 H^i_W(X,\Z(n))/ \Tor &\times H^j_W(X,\Z(u))/\Tor \to \Z\\
\Tor  H^i_W(X,\Z(n))&\times \Tor H^{j+1}_W(X,\Z(u))\to \Q/\Z.
\end{align*}
Under the finite generation conjecture, we have
$ \Tor H^i_\et(X,\Z(n))\cong \Tor H^i_W(X,\Z(n))$ for $i\not=2n+2$,
and $\Tor H^{2n+2}_\et(X,\Z(n))\cong H^{2n+2}_W(X,\Z(n))\oplus (\Q/\Z)^r$,
where $r$ is the rank of $H^{2n}_\et(X,\Z(n))$, as one sees
from the long exact sequence \cite[Thm. 7.1]{ichweil}.
\end{remark}

\subsection*{Arithmetic schemes}
The same discussion as for finite fields should also apply to 
arithmetic schemes, i.e. schemes which are regular, and proper over
the spectrum $B$ of the ring of integers of a number field 
or a smooth and proper curve over a finite field. See \cite{ichdede}
for properties of Bloch's higher Chow groups
on smooth schemes over a Dedekind ring. 
In order to get the correct $2$-torsion
in the presence of real embeddings, one has to consider cohomology with
compact support. It is defined as the \'etale cohomology with compact
support on $B$ of $Rf_!\Z(n)$, see \cite[\S 3]{kato}. 
There is an exact sequence 
\begin{equation}\label{cet}
\cdots \to H^i_c(X,\Z(n)) \to H^i_\et(X,\Z(n)) \to 
\bigoplus_{v\in S_\infty} H^i_T(\R,R\Gamma_\et(X_{\mathbb C},\Z(n)))
\to\cdots, 
\end{equation}
where the last term is Tate-modified cohomology, a finite $2$-group. 
To use the same argument as above, two ingredients are
missing, see also the discussion in \cite[\S 6]{fm}:  

1) The cup-product 
$$ H^{j}_\et(X,\Z(u))\times  H^{i+1}_c(X,\Z(n))\to H^{2d+2}_c(X,\Z(d)) 
\to H^4_c(B,\Z(1))\cong \Q/\Z$$
is conjectured to exist, but this is currently unknown for higher Chow groups
\cite{marcmc}.
The problem is that if two cycles are located in the same special fiber,
they do not intersect in the correct codimension, so that one cycle has to 
be moved to lie horizontal or in another special fiber. 
Spitzweck \cite{spitzweck} has announced a construction of this pairing
in case that $X$ is smooth over $B$.

2) Duality with finite coefficients, i.e., a perfect pairing of finite groups, 
\begin{equation}\label{fcp}
 H^{i}_\et(X,\Z/m(n))\times H^{j}_c(X,\Z/m(u))\to 
H^3_c(B,\Z/m(1))\cong \Z/m
\end{equation}
is not known to exist. If the fibers at all places dividing $m$
are normal crossing schemes, then Sato proved a duality
as above for $\Z/m(n)$ replaced by his $p$-adic Tate-twists $\T_m(n)$
\cite[Thm. 1.2.1]{sato}. It is expected that $\T_m(n)$ and 
$\Z/m(n)$ are quasi-isomorphic, but this is only know if $\Z/m(n)$ is
acyclic in degrees larger than $n$ (because $\T_m(n)$ has this property 
by construction) \cite{zhong}. This would follow, for
example, from a Gersten resolution for $\Z/m(n)$, but this
is only known for smooth schemes \cite{ichdede}. In particular,
we obtain such a pairing localized aways from all $p$ where $X$ has
bad reduction at a place above $p$.


\begin{conjecture}(Lichtenbaum) \label{licht}
The groups $H^i_\et(X,\Z(n))$ are finitely generated for $i\leq 2n$,
finite for $i=2n+1$, and of cofinite type for $i\geq 2n+2$.
\end{conjecture}

If follows from the long-exact sequence \eqref{cet} that then the
same statement holds for cohomology with compact support.
On the other hand, the statement of the conjecture is wrong 
if one removes points from the base $B$.

\begin{proposition}
Assume Conjecture \ref{licht} and the existence of the pairings
with finite coefficients \eqref{fcp}. Then we have perfect pairings
$$ H^j_\et(X,\Z(u))^\wedge  \times \Tor H^{i+1}_c(X,\Z(n))\to \Q/\Z,$$
$$ H^j_c(X,\Z(u))^\wedge  \times \Tor H^{i+1}_\et(X,\Z(n))\to \Q/\Z.$$
\end{proposition}

\begin{proof}(see also \cite[Prop. 3.4]{fm})
We show the first statement, the proof of the other statement is identical.
If $j\leq 2u$, then 
$T H^{j+1}_\et(X,\Z(u))=0$ (as the Tate module of  a finitely generated
group vanishes) and $H^i_c(X,\Z(u))\otimes \Q/\Z=0$ (as the 
cohomology group is torsion) and we obtain
$$ H^j_\et(X,\Z(u))^\wedge \cong \operatorname{lim} H^j_\et(X,\Z/m(u))
\cong H^i_c(X,\Q/\Z(n))^* \cong (\Tor H^{i+1}_c(X,\Z(n)))^*.$$
If $j>2u$, then $H^j_\et(X,\Z(u))^\wedge$ is finite, 
$(H^i_c(X,\Z(n))\otimes\Q/\Z)^*$ is torsion free, and 
$$ H^j_\et(X,\Z(u))^\wedge \cong \Tor \operatorname{lim} H^j_\et(X,\Z/m(u))
\cong \Tor (H^i_c(X,\Q/\Z(n))^*) \cong (\Tor H^{i+1}_c(X,\Z(n)))^*.$$\qed
\end{proof}

\section{Local fields}
Let $k$ be a local field of characteristic $p\geq 0$, 
i.e. a complete discrete valuation field with finite residue field. 
If we take $w=1$ and $j=3$ in \eqref{basicpair}, then 
setting 
$$n+u=d+1,\quad  i+j=2d+2,$$
we obtain a pairing 
$$ H^{j}_\et(X,\Z(u))\times H^{i+1}_\et(X,\Z(n)) \to 
H^{2d+3}_\et(X,\Z(d+1)) \stackrel{tr}{\longrightarrow} 
H^3_\et(k,\Z(1))\cong\Br k\cong \Q/\Z  .$$
Combining this with the duality over local fields for $p\not|m$ (see
\cite[I Cor. 2.3]{adt} combined with Poincar\'e duality over algebraically 
closed fields), 
$$ H^{j}_\et(X,\Z/m(u))\times H^{i}_\et(X,\Z/m(n)) \to  \Q/\Z ,$$
we again obtain an exact sequence
\begin{multline}\label{locfseq}
 0\to H^{j}_\et(X,\Z(u))/m \to ({}_m H^{i+1}_\et(X,\Z(n)))^* \\
\stackrel{\delta}{\longrightarrow}
 {}_m H^{j+1}_\et(X,\Z(u))\to (H^{i}_\et(X,\Z(n))/m)^*\to 0,
\end{multline}
and pairings which are non-degenerate on the left
$$H^{j}_\et(X,\Z(u))/m \times  {}_m H^{i+1}_\et(X,\Z(n))\to\Q/\Z.$$
If $\delta$ is the zero-map, then this induces in the limit a duality 
$$H^{j}_\et(X,\Z(u))^\wedge \times  \Tor H^{i+1}_\et(X,\Z(n))
\to \Q/\Z,$$
i.e. the duality "in some appropriate sense
of the term" expected by Lichtenbaum \cite[\S 6]{lichtenbaumMC}.

\begin{example}
The vanishing of $\delta$ for $X$ the spectrum of a local 
field is equivalent  to class field theory. 
Indeed, for $u=1, j=1$ it states that the injection
$$ H^1_\et(k,\Z(1))^\wedge \cong (k^\times )^\wedge
\to H^2_\et(k,\Z)^* \cong  H^1_\et(k,\Q/\Z)^* \cong
\Gal(k)^\ab$$ 
is an isomorphism. For $u=0, j=0$ it states that the injection
$$ H^0_\et(k,\Z)^\wedge\cong \hat \Z \to \Br(k)^* \cong 
H^3_\et(k,\Z(1))^*$$
 is an isomorphism.
\end{example}

\begin{example}
If $X$ is a curve over a $p$-adic field
and $n=1$, then the statement for $i=1,2$ is the duality between
$\Pic(X)$ and $\Br(X)$ proven by Lichtenbaum \cite{lichtenbaumcurve}. For
$i=0,3$ it follows from 
$$H^5_\et(X,\Z(1))\cong H^2(k, H^2_\et(\bar X,\Q/\Z(1))\cong H^2(k,\Q/\Z)=0.$$
\end{example}

\begin{proposition}
Assume that either $i\not\in \{n,\ldots, n+d+1\}$, or that 
$X$ has good reduction and $i\not=2n-1, 2n, 2n+1$.
Then $\delta=0$.
\end{proposition}

\begin{proof}
Again the vanishing of $\delta_m$ for all $m$ is
equivalent to the vanishing of $\delta_\infty $ in the diagram
$$\begin{CD}
(\Tor H^{i+1}_\et(X,\Z(n)))^* @>\delta_\infty >> T H^{j+1}_\et(X,\Z(u))\\
@VVV @VVV \\
({}_m H^{i+1}_\et(X,\Z(n)))^* @>\delta_m >> {}_m H^{j+1}_\et(X,\Z(u))
\end{CD}$$ 
because the left vertical map is surjective.
By \cite{kahnfre},  
$H^i_\et(X,\Q/\Z[\frac{1}{p}](n))$ is finite for 
$i\not\in \{n,\ldots , n+d+1\}$ for
general $X$, and $i\not=2n-1,2n,2n+1$ for $X$ with good reduction. 
This implies that 
$\Tor H^{i+1}_\et(X,\Z(n))$ is finite for $i<n$ and $i< 2n-1$,
respectively, hence its dual cannot map non-trivially to the torsion free
Tate-module. On the other hand,  
the Tate module $T H^{j+1}_\et(X,\Z(u))$ vanishes for 
$j<u\Leftrightarrow i> n+d+1$ and 
$j< 2u-1\Leftrightarrow i> 2n+1$, respectively.\qed 
\end{proof}

We believe that an improvement is possible:

\begin{conjecture}
If $i\not\in \{ n+1,\ldots, n+d\} $ or if 
$X$ has good reduction and $i\not=2n$, then $\delta=0$.
\end{conjecture}

We give examples for $\delta $ to be non-zero, thus 
giving counterexamples to duality of 
\'etale motivic cohomology over local fields. 
Since $H^1_\et(X,\Z)=0$, we get from the limit of 
\eqref{locfseq} a short exact sequence 
$$ 0\to H^{2d+1}_\et(X,\Z(d+1))^\wedge \to (\Tor H^2_\et(X,\Z))^* \to 
T H^{2d+2}_\et(X,\Z(d+1))\to 0.$$
For $X$ a curve, $H^{3}_\et(X,\Z(2))\cong H^{3}_\M(X,\Z(2))\cong 
SK_1(X)$, and it follows from 
S.\ Saito's result \cite[Thm. 2.6]{saitocurve}
that the right hand side has rank equal to 
$\rk H^1_\et(Y,\Z)$, where $Y$ the special fiber of a smooth and proper model. 
For arbitrary dimension, Yoshida proved that its rank 
is the dimension
of the maximal split torus of the Neron model of $\Alb_X$ \cite{yoshida}.
Hence $\delta $ can be non-zero for a curve 
(with bad reduction) in weights $n=2, u=0$. 

We now give an example, obtained with the help of S.\ Saito and K.\ Sato, 
showing that $\delta $
can be non-zero even for schemes with good reduction.
For $ n=d, i=2d$, the limit of \eqref{locfseq} gives a sequence
\begin{equation*}
0\to \Pic(X)^{\wedge l}\to (H^{2d+1}_\et(X,\Z(d))\{l\})^* 
\stackrel{\delta}{\to} T_l \Br(X)\to (H^{2d}_\et(X,\Z(d))\otimes\Q_l/\Z_l)^*\to 0.
\end{equation*}

\begin{proposition}
Assume that $X$ admits a smooth and proper model $\mathcal X$.
Then we have a commutative diagram with exact rows:
$$ \begin{CD}
0@>>> T_l \Br \mathcal X@>>> T_l\Br X@>>> (CH_0(X)\otimes \Q_l/\Z_l)^*\\
@.@VVV @VVV @|\\
@.0 @>>> (H^{2d}_\et(X,\Z(d))\otimes\Q_l/\Z_l)^* @>>> (CH_0(X)\otimes \Q_l/\Z_l)^*
\end{CD}$$
Here the upper row is induced by the Brauer-Manin pairing and the lower row
by the change of topology map.
\end{proposition}

\begin{proof} 
By Colliot-Th\'el\`ene and Saito \cite[Cor. 2.4]{its}, 
the kernel of the Brauer-Manin 
pairing is $ T_l \Br \mathcal X$. The lower row is exact because the 
cokernel of $CH_0(X)\to H^{2d}_\et(X,\Z(d))$ is torsion, hence it vanishes
after tensoring with $\Q/\Z$. It remains to show 
that the diagram is commutative.
\end{proof}

\begin{lemma}
The following diagram is commutative, where the upper pairing is the
Brauer-Manin pairing and the lower pairing the cup-product pairing:
$$\begin{CD}
CH_0(X)/m  @. \times @. {}_m \Br(X) @>>> \Z/m \\
@VVV @. @|  @|\\
H^{2d}_\et(X,\Z(d))/m @. \times @. {}_m \Br(X) @>>> \Z/m .
\end{CD}$$
\end{lemma}

\begin{proof}
By definition of the lower pairing, 
it suffices to show this after adding the following commutative diagram 
on the bottom
$$\begin{CD}
H^{2d}_\et(X,\Z(d))/m @. \times @. {}_m \Br(X) @>>> \Z/m \\
@VVV @. @AAA   @|\\
H^{2d}_\et(X,\Z/m(d)) @. \times @. H^2_\et(X,\Z/m(1) @>>> \Z/m 
\end{CD}$$
because the middle vertical map is surjective. The Brauer-Manin
pairing is defined point by point. But if $i:\Spec k\to X$ is a closed
point, then the commutativity follows from the projection formula,
i.e. the commutativity of the following diagram
$$\begin{CD}
H^{0}_\et(k,\Z/m(0)) @. \times @. H^2_\et(k,\Z/m(1)) @>>> 
\Z/m\cong H^2_\et(k,\Z/m(1) \\
@Vi_* VV @. @Ai^* AA   @V i_* V\sim V\\
H^{2d}_\et(X,\Z/m(d)) @. \times @. H^2_\et(X,\Z/m(1)) @>>> \Z/m\cong 
 H^{2d+2}_\et(X,\Z/m(d+1)) .
\end{CD}$$\qed
\end{proof}

Finally, we note that there are examples with non-vanishing
$ T_l \Br \mathcal X$: Let $X/\Q_p$ be the self product
of an elliptic curve $E/\Q_p$ without complex multiplication and 
good reduction $E_s$.
Then the graph of the Frobenius of $E_s$ in $\Pic(Y)$, $Y$ the special fiber of
the proper smooth model $\mathcal X/\Z_p$, does not lift to $\Pic(\mathcal X)$, 
because $\End(E)$ has rank $1$ \cite[p. 331]{LS}. 
Hence we conclude by the proper base change theorem and the following diagram
$$\begin{CD}
0@>>> \Pic(\mathcal X)^{\wedge l} @>>>H^2(\mathcal X,\Z_l(1))  @>>> T_l \Br \mathcal X@>>> 0\\
@. @VVV @| @VVV \\
0@>>> \Pic(Y)^{\wedge l} @>>>H^2(Y,\Z_l(1))  @>>> T_l \Br Y@= 0
\end{CD}$$


\begin{remark}
We believe that there should be a better behaved duality theory for Weil-\'etale
cohomology groups, see \cite{karpuk} in the case of curves. B. Morin expects
that there are locally compact groups $H^i_W(X,\Z(n))$ and $H^j_W(X,\R/\Z(u))$,
together with a trace map 
$H^{2d+2}_W(X,\R/\Z(d+1))\to H^{2}_W(K,\R/\Z(1))\cong \R/\Z$,
such that there is a perfect Pontrjagin pairing of locally compact groups
$$ H^i_W(X,\Z(n))\times H^{2d+2-i}_W(X,\R/\Z(d+1-n))\to 
H^{2d+2}_W(X,\R/\Z(d+1))\to \R/\Z.$$
\end{remark}

\small

\myaddress{Department of Mathematics, Rikkyo University,\\
Nishi-ikebukuro, Toshimaku, Tokyo, Japan}
\myemail{\tt geisser@rikkyo.ac.jp}

\end{document}